\documentclass[12pt, a4paper, draft]{amsart}

\usepackage{fullpage}
\usepackage{amsmath,amssymb, amsthm}
\usepackage{xypic}
\usepackage[utf8]{inputenc}

\usepackage{color}
\definecolor{webcolor}{rgb}{0,0,1}
\definecolor{webbrown}{rgb}{.6,0,0}
\usepackage[
        colorlinks,
        linkcolor=webbrown, filecolor=webbrown,  citecolor=webbrown, 
        pdfauthor={},
  pdftitle={},
]{hyperref}
\usepackage{enumitem}

\newcommand{\N}{\mathbf{N}}
\newcommand{\Z}{\mathbf{Z}} 
\newcommand{\Q}{\mathbf{Q}}

\newcommand{\PP}{\mathbf{P}}
\DeclareMathOperator{\Spec}{Spec}

\DeclareMathOperator{\GCD}{GCD}

\theoremstyle{definition}
\newtheorem{theorem}{Theorem}[section]
\newtheorem{algorithm}[theorem]{Algorithm}
\newtheorem{definition}[theorem]{Definition}
\newtheorem{lemma}[theorem]{Lemma}
\newtheorem{prop}[theorem]{Proposition}
\newtheorem{corollary}[theorem]{Corollary}

\newtheorem{example}[theorem]{Example}
\newtheorem{remark}[theorem]{Remark}

\newtheorem*{theorem*}{Theorem}
\newtheorem*{notation*}{Notation}

\newcommand{\lcm}{\ensuremath{\mathrm{lcm}}}

\DeclareMathOperator{\Syz}{Syz}
\DeclareMathOperator{\LM}{lm}
\DeclareMathOperator{\LC}{lc}
\DeclareMathOperator{\LT}{lt}
\DeclareMathOperator{\Lt}{Lt}
\DeclareMathOperator{\spoly}{spoly}

\title[Gröbner bases over Dedekind domains]{On Gröbner bases over Dedekind domains}

\author{Tommy Hofmann}
\address{Tommy Hofmann\\
Fachbereich Mathematik\\
Technische Universität Kaiserslautern\\
67663 Kaiserslautern\\
Germany}
\email{thofmann@mathematik.uni-kl.de}
\urladdr{http://www.mathematik.uni-kl.de/$\sim$thofmann}

\subjclass[2010]{13P10}
\keywords{}
\date{\today}

\begin{document}

\begin{abstract}
  Gröbner bases are a fundamental tool when studying ideals in multivariate
  polynomial rings. More recently there has been a growing interest in
  transferring techniques from the field case to other
  coefficient rings, most notably Euclidean domains and principal ideal rings.
  In this paper we will consider multivariate polynomial rings over Dedekind
  domain. By generalizing methods from the theory of finitely generated
  projective modules, we show that it is possible to describe Gröbner bases
  over Dedekind domains in a way similar to the case of principal ideal
  domains, both from a theoretical and algorithmic point of view.
\end{abstract}

\maketitle

\section{Introduction}

The theory of Gröbner bases, initiated by Buchberger~\cite{Buchberger1965} plays an
important role not only in mathematical disciplines like algorithmic commutative
algebra and algebraic geometry, but also in related areas of science and
engineering.  Although the original approach of Buchberger was restricted to
multivariate polynomials with coefficients in a field,
Trinks~\cite{Trinks1978} and Zacharias~\cite{Zacharias1978} showed that by
generalizing the notions of $S$-polynomials and reduction, Gröbner bases can
also be constructed in the ring case.  For coefficient rings that are principal
ideal domains, the approach to constructing Gröbner bases is very close to the
field case has attracted a lot of attention, see for example \cite{Pauer1988, KandriRody1988, Moller1988, Pan1989}, also \cite[Chapter 4]{Adams1994} or \cite[Chapter 10]{Becker1993}.

In this paper, we will investigate Gröbner bases over Dedekind domains, that
is, over integral domains which are locally discrete valuation rings.  Despite
the prominent role of Dedekind domains as coefficient rings for example in
arithmetic geometry, not much is known in connection with the construction of
Gröbner bases. Our aim is to show that it is possible to improve upon the
generic algorithms for Noetherian domains. In particular, we will show that using
the notions of pseudo-polynomials and pseudo-Gröbner bases the approach comes
very close to that of principal ideal domains.

The idea of using so called pseudo-objects to interpolate between principal
ideal domains and Dedekind domains has already been successfully applied to the
theory of finitely generated projective modules.  Recall that over a principal
ideal domain such modules are in fact free of finite rank. By using the Hermite
and Smith form, working with such modules is as easy as working with finite dimensional
vector spaces over a field. If the ring is merely a Dedekind domain, such
modules are in general not free, rendering the Hermite and Smith form useless. But
since the work of Steinitz~\cite{Steinitz1911, Steinitz1912} it has been known that these modules
are direct sums of projective submodules of rank one. In \cite{Cohen1996} (see
also \cite{Cohen2000}), based upon ideas already present in
\cite{Bosma1991}, a theory of pseudo-elements has been developed, which
enables an algorithmic treatment of this class of modules very close to the
case of principal ideal domains. In particular, a generalized Hermite form
algorithm is described, which allows for similar improvements as the classical
Hermite form algorithm in the principal ideal case, see also \cite{Biasse2017, Fieker2014}.

Now---in contrast to the setting of finitely generated projected modules just described---Gröbner bases do exist if the coefficient ring is a
Dedekind domain.
In~\cite{Adams1997} using a generalized version of Gröbner basis, the structure of ideals in univariate polynomial rings over Dedekind domains is studied.
Apart from that, nothing is published on how to exploit the structure of Dedekind
domains in the algorithmic study of ideals in multivariate polynomial rings.
Building upon the notion of
pseudo-objects, in this paper we will introduce pseudo-Gröbner bases, that will
interpolate more smoothly between the theory of Gröbner bases for Dedekind
domains and principal ideal domains.
Of course the hope is that one can apply more
sophisticated techniques from principal ideal domains to Dedekind domains, for example, signature-based algorithms as introduced in \cite{Eder2017}.
As an illustration of this idea, we prove a simple generalization of the product criterion for pseudo-polynomials.
We will also show how to use the pseudo-Gröbner basis to solve basic tasks from algorithmic commutative algebra,
including the computation of primes of bad reduction.

The paper is organized as follows.
In Section~2 we recall standard notions from multivariate polynomials and translate
them to the context of pseudo-polynomials.
This is followed by a generalization of Gröbner bases in Section~3, where we
present various characterizations of the so called pseudo-Gröbner bases.
In Section~4 by analyzing syzygies of pseudo-polynomials, we prove a variation of Buchberger's criterion.
As a result we obtain a simple to formulate algorithm for computing Gröbner bases.
We also use this syzygy-based approach to prove the generalized product criterion.
In Section~5 we consider the situation over a ring of integers of a number
field and address the omnipresent problem of quickly growing coefficients by
employing classical tools from algorithmic number theory.
In the final section we give some applications to classical problems in algorithmic
commutative algebra and the computation of primes of bad reduction.

\subsection*{Acknowledgments}
The author was supported by Project II.2 of SFB-TRR 195 `Symbolic Tools in
Mathematics and their Application'
of the German Research Foundation (DFG).

\begin{notation*}
Throughout this paper, we will use $R$ to denote a Dedekind domain, that is, a Noetherian integrally closed domain of Krull dimension one, and $K$ to denote its total ring of fractions.
  Furthermore, we fix a multivariate ring $R[x] = R[x_1,\dotsc,x_n]$ and a monomial ordering $<$ on $R[x]$.
\end{notation*}

\section{Pseudo-elements and pseudo-polynomials}

In this section we recall basic notions from multivariate polynomials and
generalize them in the context of pseudo-polynomials over Dedekind domains.

\subsection{Multivariate polynomials.}

For $\alpha = (\alpha_1,\dotsc,\alpha_n) \in \N^n$, we denote by $x^\alpha$ the
monomial $x_1^{\alpha_1}\dotsm x_n^{\alpha_n}$. We call $\alpha$ the \textit{degree} of $f$ and denote it by $\deg(f)$.
A polynomial $f = c x^\alpha$ with $c \in R$ and $\alpha \in \N^n$ is called a \textit{term}.
For an arbitrary multivariate polynomial $f \in \sum_{\alpha \in \N^n} c_\alpha x^\alpha$ we denote by
$\deg(f) = \max_>\{ \alpha \in \N^n \mid c_\alpha \neq 0\}$ the \textit{degree} of $f$, by $\LM(f) = x^{\deg(f)}$ the \textit{leading monomial}, by $\LC(f) = c_{\deg(\alpha)}$ the \textit{leading coefficient} and by $\LT(f) = c_{\deg(f)}x^{\deg(f)}$ the \textit{leading term} of $f$.

\subsection{Pseudo-elements and pseudo-polynomials.}

A fractional ideal of $R$ is a non-zero finitely generated $R$-submodule of $K$.
Let now $V$ be a vector space over $K$ and $M$ an $R$-submodule of $V$ such that $KM = V$, that is, $M$ contains a $K$-basis of $V$.
Given a fractional ideal $\mathfrak a$ of $R$ and an element $v \in V$ we denote
by $\mathfrak av$ the set $\{ \alpha v \mid \alpha \in \mathfrak a \} \subseteq V$,
which is in fact an $R$-submodule of $V$.

\begin{definition}
A pair $(v, \mathfrak a)$ consisting of an element $v \in V$ and a fractional ideal $\mathfrak a$ of $R$ is called
  a \textit{pseudo-element} of $V$. In case  $\mathfrak a v \subseteq M$, we call $(v, \mathfrak a)$ a \textit{pseudo-element} of $M$.
\end{definition}

\begin{remark}
  The notion of pseudo-objects goes back to Cohen~\cite{Cohen1996}, who introduced them to compute with finitely generated projective modules over Dedekind domains.
  Note that in~\cite{Cohen2000} the $R$-submodule $\mathfrak a v$ itself is
  defined to be a pseudo-element, whereas with our definition, this
  $R$-submodule is only attached to the pseudo-element $(v, \mathfrak a)$.
  We choose the slightly modified version to simplify the exposition and to ease notation.
\end{remark}

\begin{lemma}\label{lem:lem1}
  Let $V$ be a $K$-vector space.
  \begin{enumerate}
    \item
      For $v, w \in V$ and $\mathfrak a, \mathfrak b, \mathfrak c$ fractional ideals of $R$ we have
      $\mathfrak a(\mathfrak b v) = (\mathfrak a \mathfrak b)v$ and $\mathfrak c(\mathfrak av + \mathfrak bw) = (\mathfrak c \mathfrak a)v + (\mathfrak c \mathfrak b)vw$.
    \item
      Let $(v, \mathfrak a)$, $(v_i, \mathfrak a_i)_{1 \leq i \leq l}$ be
      pseudo-elements of $V$. If $\mathfrak av \subseteq \sum_{1 \leq i
      \leq l} \mathfrak a_i v_i$, then there
      exist $a_i \in \mathfrak a_i\mathfrak a^{-1}$, $1 \leq i \leq l$, such that
      $v = \sum_{1 \leq i \leq l} a_i v_i$.
  \end{enumerate}
\end{lemma}

\begin{proof}
  (i): Clear. (ii): Using (i) and by multiplying with $\mathfrak a^{-1}$ we are reduced to the case where $\mathfrak a = R$, that is,
  $v \in \sum_{1 \leq i \leq l} \mathfrak a_i v_i$. But then the assertion is clear.
\end{proof}

We will now specialize to the situation of multivariate polynomial rings, where additionally we have the $R[x]$-module structure.
For a fractional ideal $\mathfrak a$ of $R$ we will denote by $\mathfrak a[x]$ the ideal $\{ \sum_{\alpha} c_\alpha x^\alpha \mid c_\alpha \in \mathfrak a\}$ of $R[x]$.

\begin{lemma}\label{lem:rep}
The following hold:
  \begin{enumerate}
    \item
      For fractional ideals $\mathfrak a, \mathfrak b, \mathfrak c$ of $R$ we have $\mathfrak a(\mathfrak b[x]) = (\mathfrak a \mathfrak b)[x]$ and $\mathfrak a(\mathfrak b[x] + \mathfrak c[x]) = (\mathfrak a\mathfrak b)[x] + (\mathfrak a \mathfrak c)[x]$.
    \item
      If $M$ is an $R[x]$-module and $(v, \mathfrak a)$ a pseudo-element, then $\langle \mathfrak a v \rangle_{R[x]} = \mathfrak a[x]v$.
    \item
      Let $M$ be an $R[x]$-module and $(v_i,\mathfrak a_i)_{1 \leq i \leq l}$ pseudo-elements of $M$ with $\langle \mathfrak a_i v_i \mid 1 \leq i \leq l\rangle_{R[x]}  = M$. Given a pseudo-element $(v, \mathfrak a)$ of $M$, there exist $f_i \in \mathfrak a_i \mathfrak a^{-1}[x]$, $1 \leq i \leq l$, such that $v = \sum_{1 \leq i \leq l} f_i v_i$.
  \end{enumerate}
\end{lemma}

\begin{proof}
  Item~(i) follows from the distributive properties of ideal multiplication. Proving (ii), (iii) is analogous to Lemma~\ref{lem:lem1}.
\end{proof}

\begin{definition}
  A \textit{pseudo-polynomial} of $R[x]$ is a pseudo-element of $R[x]$, that is, a pair $(f, \mathfrak f)$ consisting of a polynomial $f \in K[x]$ and a fractional ideal $\mathfrak f$ of $R$ such that $\mathfrak f \cdot f \subseteq R[x]$.
  We call $\mathfrak f \LC(f) \subseteq R$ the \textit{leading coefficient} of $(f, \mathfrak f)$ and denote it by $\LC(f, \mathfrak f)$.
  The set $\mathfrak f[x] f \subseteq R[x]$ is called the \textit{ideal generated} by $(f, \mathfrak f)$ and is denoted by $\langle(f, \mathfrak f)\rangle$.
  We say that the pseudo-polynomial $(f, \mathfrak f)$ is \textit{zero}, if $f = 0$.
\end{definition}

\begin{lemma}
  Let $(f, \mathfrak f)$ be a pseudo-polynomial of $R[x]$. Then the following hold:
  \begin{enumerate}
    \item
      The leading coefficient $\LC(f, \mathfrak f)$ is an integral ideal of $R$.
    \item
      We have $\langle \mathfrak f f \rangle_{R[x]} = \mathfrak f[x] f$.
  \end{enumerate}
\end{lemma}

\begin{proof}
  Clear.
\end{proof}


\section{Reduction and pseudo-Gröbner bases}

At the heart of the construction of Gröbner bases lies a generalization of the Euclidean division in univariate polynomial rings.
In the context of pseudo-polynomials this takes the following form.

\begin{definition}[Reduction]
  Let $(f, \mathfrak f)$ and $G = \{ (g_i, \mathfrak g_i) \mid 1 \leq i \leq l\}$ be set of non-zero pseudo-polynomials of $R[x]$ and $J = \{ 1 \leq i \leq l \mid \LM(g_i) \text{ divides } \LM(f) \}$.
  We say that $(f, \mathfrak f)$ \textit{can be reduced modulo} $G$ if $\LC(f, \mathfrak f) \subseteq \sum_{i \in J} \LC(g_i,
\mathfrak g_i)$.
  In case $G = \{(g, \mathfrak g)\}$ consists of a single pseudo-polynomial, we say that $(f, \mathfrak f)$ \textit{can be reduced modulo} $(g, \mathfrak g)$.
  We define $(f, \mathfrak f)$ to be \textit{minimal with respect to $G$}, if it cannot be reduced modulo $G$.
\end{definition}

\begin{lemma}\label{lem:red}
  Let $(f, \mathfrak f)$ and $G =  \{ (g_i, \mathfrak g_i) \mid 1 \leq i \leq l\}$ be
  non-zero pseudo-polynomials of $R[x]$ and $J = \{ 1 \leq i \leq l \mid \LM(g_i) \text{ divides }
  \LM(f) \}$. Then $(f, \mathfrak f)$ can be reduced modulo $G$ if and only if there exist
  $a_i \in \mathfrak g_i \mathfrak f^{-1}$, $i \in J$, such that $\LC(f) =
  \sum_{i \in J} a_i \LC(g_i)$.
\end{lemma}

\begin{proof} Set $\mathfrak c = \sum_{i \in J} \LC(g_i, \mathfrak g_i)$.
  First assume that $(f, \mathfrak f)$ can be reduced modulo $G$, that is $\LC(f, \mathfrak f) = \mathfrak f \LC(f) \subseteq \mathfrak c$.
  Hence $\LC(f) \in \mathfrak c \mathfrak f^{-1} = \sum_{i \in J} \LC(g_i,\mathfrak
  g_i) \mathfrak f^{-1}$ and there exist $b_i \in \mathfrak g_i
  \mathfrak f^{-1}\LC(g_i)$, $i \in J$, such that $\LC(f) = \sum_{i \in J}
  b_i$. Then the elements $a_i = b_i/\LC(g_i) \in \mathfrak g_i \mathfrak
  f^{-1}$, $i \in J$, satisfy the claim.

  On the other hand, if $\LC(f) = \sum_{i \in J} \alpha \LC(g_i)$ for $a_i \in \mathfrak g_i \mathfrak f^{-1}$, then
  \[ \LC(f, \mathfrak f) = \mathfrak f \LC(f) \subseteq \sum_{i \in J} \mathfrak f a_i \LC(g_i) \subseteq \sum_{i \in J} \LC(g, \mathfrak g_i). \qedhere \]
\end{proof}

\begin{lemma}\label{lem:div}
  Let $(f, \mathfrak f)$ and $(g, \mathfrak g)$ be two non-zero pseudo-polynomials of $R[x]$.
  Then the following are equivalent:
  \begin{enumerate}
    \item
      $(f, \mathfrak f)$ can be reduced modulo $(g, \mathfrak g)$.
    \item
      $\mathfrak f[x] \LT(f) \subseteq \mathfrak g[x] \LT(g)$,
    \item
      $\mathfrak f \LC(f) \subseteq \mathfrak g \LC(g)$ and $\LM(f)$ divides $\LM(g)$.
  \end{enumerate}
\end{lemma}

\begin{proof}
  (i) $\Rightarrow$ (ii): By assumption $\LM(g)$ divides $\LM(f)$ and $\LC(f) = \alpha \LC(g)$ for some $\alpha \in \mathfrak g \mathfrak f^{-1}$.
  Hence
  \[ \mathfrak f[x] \LT(f) = \mathfrak f[x] \LC(f) x^{\deg(f)} = \mathfrak f[x] \alpha x^{\deg(f) - \deg(g)} \LT(g) \subseteq \mathfrak f[x] \alpha \LT(g) \subseteq \mathfrak g[x] \LT(g) .\]
  (ii) $\Rightarrow$ (iii): Let $\mu \in \mathfrak \subseteq \mathfrak f[x]$. Since $\mu \LT(f) \in \mathfrak g[x] \LT(g)$ it follows that $\LM(g)$ divides $\LM(f)$ and $\mathfrak f \LC(f) \subseteq \mathfrak g \LC(g)$.
  (iii) $\Rightarrow$ (i): Clear.
\end{proof}

\begin{definition}
  Let $(f, \mathfrak f)$ and $G = \{ (g_i, \mathfrak g_i) \mid 1 \leq i \leq l\}$ be pseudo-polynomials of $R[x]$ and assume that $(f,\mathfrak f)$ can be reduced modulo $G$ and $(a_i)_{i \in J}$ are as in Lemma~\ref{lem:red}. Then we call
  $(f - \sum_{i \in J} a_i g_i, \mathfrak f)$ a \textit{one step reduction} of $(f, \mathfrak f)$ with respect to $G$ and
  we write 
  \[ (f, \mathfrak f) \xrightarrow{G} \Bigl(f - \sum_{i \in J} a_i x^{\deg(f) - \deg(g_i)}  g_i, \mathfrak f\Bigr). \]
\end{definition}

\begin{lemma}\label{lem:propred}
  Let $(h, \mathfrak f)$ be a one step reduction of $(f, \mathfrak f)$ with respect to $G = \{ (g_i,\mathfrak g_i) \mid 1 \leq i \leq l\}$.
  Denote by $I = \langle G \rangle$ the ideal of $R[x]$ generated by $G$. Then the
  following hold:
  \begin{enumerate}
    \item
      The pair $(h, \mathfrak f)$ is a pseudo-polynomial of $R[x]$.
    \item
      We have $\mathfrak f[x](f - h) \subseteq I$.
    \item
      We have $\langle (f, \mathfrak f) \rangle \subseteq I$ if and only if $\langle (h, \mathfrak f) \rangle \subseteq I$.
  \end{enumerate}
\end{lemma}

\begin{proof}
  By definition there exists $J \subseteq \{1,\dotsc,r\}$, $a_i \in \mathfrak g_i \mathfrak f^{-1}$, $i \in J$, with $\LC(f) = \sum_{i\in J} a_i \LC(g_i)$.

  (i): We have
  \[ \mathfrak f h = \mathfrak f\Bigl(f - \sum_{i \in J} a_i g_i\Bigr) \subseteq \mathfrak f f + \sum_{i \in J} \mathfrak f a_i g_i \subseteq \mathfrak f f + \sum_{i \in J} \mathfrak g_i g_i \subseteq R[x]. \]

  (ii): Since $f - h = \sum_{i \in I} a_i g_i$ and $a_i \in \mathfrak g_i \mathfrak f^{-1}$ it is clear that $\mathfrak f a_i g_i \subseteq I$.

  (iii): If $\mathfrak f[x]f \subseteq I$, then $\mathfrak f[x](f - \sum_{i \in J} a_i g_i) \subseteq I$, since $\mathfrak f a_i \subseteq \mathfrak g_i$.
  On the other hand, if $\mathfrak f[x] f \subseteq I$, then
  \[ \mathfrak f[x] f = \mathfrak f[x]\Bigl(f - \sum_{i \in J}a_i g_i + \sum_{i \in J} a_i g_i\Bigr) \subseteq \mathfrak f[x]f + \mathfrak f[x]\Bigl(\sum_{i \in J} a_i g_i\Bigr)  \subseteq I.\qedhere \]
\end{proof}

\begin{definition}
  Let $(f, \mathfrak f)$, $(h, \mathfrak f)$ and $G = \{(g_i, \mathfrak g_i) \mid 1 \leq i \leq l\}$ be non-zero pseudo-polynomials of $R[x]$. We say that
  \textit{$(f, \mathfrak f)$ reduces to $(h, \mathfrak f)$ modulo $G$} if there exist pseudo-polynomials $(h_i, \mathfrak f)$, $1 \leq i \leq l$  such that
  \[ (f, \mathfrak f) = (h_1,\mathfrak f) \xrightarrow{G} (h_2,\mathfrak f) \xrightarrow{G} \dotsb \xrightarrow{G} (h_l,\mathfrak f) = (h,\mathfrak f). \]
  In this case we write $(f,\mathfrak f) \xrightarrow{G}_+ (h,\mathfrak f)$. (The relation $\xrightarrow{G}_+$ is thus the reflexive closure of $\xrightarrow{G}$.)
\end{definition}

\begin{lemma}\label{lem:2}
  If $(f, \mathfrak f)$, $(h, \mathfrak f)$ and $G = \{(g_i, \mathfrak g_i) \mid 1 \leq i \leq l\}$ are non-zero pseudo-polynomials with
  $(f, \mathfrak f) \xrightarrow{G} (h, \mathfrak f)$,
  then $\mathfrak f[x](f - h) \subseteq I$. Moreover $\langle (f, \mathfrak f) \rangle \subseteq I$ if and only if $\langle (h, \mathfrak f ) \rangle \subseteq I$.
\end{lemma}

\begin{proof}
  Note that $f - h = f - h_1 + h_1 - h_2 + \dotsb - h_l - h$. Hence the claim follows from Lemma~\ref{lem:propred}~(ii).
\end{proof}

\begin{remark}\label{rem:1}
  If $(h, \mathfrak f)$ is a one step reduction of $(f, \mathfrak f)$, then $\deg(h) < \deg(f)$ and there exist terms $h_i \in (\mathfrak g_i \mathfrak f^{-1})[x]$, $i \in I$, such that $f - h = \sum_{1 \leq i \leq l} h_i g_i$.
  Applying this iteratively we see that if $(h, \mathfrak f)$ is a pseudo-polynomial of $R[x]$ with $(f, \mathfrak f) \xrightarrow{G}_+ (h, \mathfrak f)$, then $\deg(h) < \deg(f)$ and there exists $h_i \in (\mathfrak g_i \mathfrak f^{-1})[x]$, $i \in I$, such that $f - h = \sum_{1 \leq i \leq} h_i g_i$.
  Moreover in both cases we have $\deg(f) = \max_{i \in I} (\deg(h_i g_i))$.
\end{remark}

\begin{definition}
  Let $(f, \mathfrak f)$ and $G = \{ (g_i, \mathfrak g_i) \mid 1 \leq i \leq l\}$ be pseudo-polynomials.
  The \textit{leading term} $\LT(f, \mathfrak f)$ is defined to be $\mathfrak f \LT(f)$.  
  Moreover we define the \textit{leading term ideal} of $(f, \mathfrak f)$ and $G$ as $\Lt(f, \mathfrak f) = \mathfrak f[x]\LT(f) = \langle \LT(f, \mathfrak f) \rangle_{R[x]}$ and $\Lt(G) = \sum_{i = 1}^r \Lt(g_i,\mathfrak g_i) \rangle_{R[x]}$ respectively.
  If $F \subseteq R[x]$ is a set of polynomials, then we define $\Lt(F) = \langle \LT(f) \mid f \in F \rangle_{R[x]}$.
\end{definition}

We can now characterize minimality in terms of leading term ideals.

\begin{lemma}\label{lem:3}
  Let $G = \{ (g_i, \mathfrak g_i) \mid 1 \leq i \leq l\}$ be non-zero pseudo-polynomials of $R[x]$.
  A non-zero pseudo-polynomial $(f, \mathfrak f)$ is minimal with respect to $G$, if and only if $\Lt(f,\mathfrak f) \not\subseteq \Lt(G)$.
\end{lemma}

\begin{proof}
  Denote by $J = \{ i \in \{1,\dotsc,r\} \mid \LM(g_i) \text{ divides } \LM(f) \}$.
  Assume first that $(f, \mathfrak f)$ is not minimal, that is, the pseudo-polynomial can be reduced modulo $G$.
  Then there exist $a_i \in \mathfrak g_i \mathfrak f^{-1}$, $i \in J$, such that $\LC(f) = \sum_{i \in J} a_i \LC(g_i)$.
  For every $i \in J$ there exists a monomial $x^{a_i}$ with $\LM(g_i) x^{a_i} = \LM(f)$.
  Hence
  \[ \LT(f) = \LM(f)\LC(f) = \sum_{i \in J} a_i \LM(f) \LC(g_i) = \sum_{ \in J} a_i x^{\alpha_i} \LT(g_i). \]
  Thus it holds that $\mathfrak f \LT(f) = \sum_{i \in J} \mathfrak f a_i x^{\alpha_i} \LT(g_i) \in \sum_{i \in J} \mathfrak g_i[x] \LT(g_i) \subseteq \Lt(G)$.
  This implies $\Lt(f, \mathfrak f) = \mathfrak f[x] f \subseteq \Lt(G)$, as claimed.

  Now assume that $\Lt(f, \mathfrak f) \subseteq \Lt(G)$.
  Let $\alpha \in \mathfrak f$. Since $\alpha \LT(f) \in \Lt(f, \mathfrak f) \subseteq \Lt(G)$, there exist
  $h_i \in \mathfrak g_i[x]$, $1 \leq i \leq l$, with $\alpha \LT(f) = \sum_{i = 1}^r h_i \LT(g_i)$.
  Without loss of generality we may assume that $h_i$ is a term, say, $h_i = a_i x^{\alpha_i}$, where $a_i \in \mathfrak g_i$.
  Denote by $J'$ the set $\{i \in \{1,\dotsc,r\} \mid x^{\alpha_i} \LM(g_i) = \LM(f_i)\}$. Hence we have
  \[ \alpha \LM(f) = \sum_{i \in J'} a_i x^{\alpha_i} \LT(g_i) = \sum_{i \in J'} a_i x^{\alpha_i} \LM(g_i) \LC(g_i). \]
  Comparing coefficients this yields $\alpha \LC(f) = \sum_{i \in J'} a_i \LC(g_i)$.
  Thus $\LC(f, \mathfrak f) = \mathfrak f \LC(f) \subseteq \sum_{i \in J'} \mathfrak g_i \LC(g_i) = \sum_{i \in J'} \LC(g_i, \mathfrak g_i)$.
  As $J' \subseteq J$, it follows that $(f, \mathfrak f)$ can be reduced modulo $G$.
\end{proof}

\begin{theorem}\label{thm:rep}
  Let $(f, \mathfrak f)$ and $G = \{ (g_i,\mathfrak g_i) \mid 1 \leq i \leq l\}$ be pseudo-polynomials.
      There exists a pseudo-polynomial $(h, \mathfrak f)$ which is minimal with respect to $G$ and $h_i \in (\mathfrak g_i \mathfrak f^{-1})[x]$, $1 \leq i \leq l$, such that $(f, \mathfrak f) \xrightarrow{G}_+ (r, \mathfrak f)$,
  \[ f - r = \sum_{i=1}^r h_i g_i, \]
      and $\deg(f) = \max((\max_{1 \leq i \leq l}\deg(h_i g_i), \deg(r))$.
\end{theorem}

\begin{proof}
  Follows immediately from Lemma~\ref{lem:3} and Remark~\ref{rem:1}.
\end{proof}

We can now generalize the characterization of Gröbner bases to pseudo-Gröbner bases.

\begin{theorem}\label{thm:defgroeb}
  Let $I$ be an ideal of $R[x]$ and $G = \{ (g_i, \mathfrak g_i) \mid 1 \leq i \leq l\}$ non-zero pseudo-polynomials of $I$. Then the following are equivalent:
  \begin{enumerate}
    \item
      $\Lt(I) = \Lt(G)$;
    \item
      For a pseudo-polynomial $(f, \mathfrak f)$ of $R[x]$ we have
      $\langle (f, \mathfrak f) \rangle \subseteq I$ if and only if $(f, \mathfrak f)$ reduces to $0$ modulo $G$.
    \item
      For every pseudo-polynomial $(f, \mathfrak f)$ of $R[x]$ with $\langle (f, \mathfrak f) \rangle \subseteq I$ there exist $h_i \in (\mathfrak g_i \mathfrak f^{-1})[x]$, $1 \leq i \leq l$, such that $f = \sum_{i=1}^r h_i g_i$ and $\LM(f) = \max_{1 \leq i \leq l}(\LM(h_ig_i))$.
    \item
      If $(a_{ij})_{1 \leq j \leq n_i}$ are ideal generators of $\mathfrak g_i$ for $1 \leq i \leq l$, 
      then the set 
      \[ \{ a_{ij}g_i \mid 1 \leq i \leq l, 1 \leq j \leq n_i\} \]
      is a Gröbner basis of $I$.
  \end{enumerate}
\end{theorem}

\begin{proof}
  (i) $\Rightarrow$ (ii): If $(f, \mathfrak f) \xrightarrow{G}_+ 0$, then Lemma~\ref{lem:2} implies that $\langle( f, \mathfrak f)\rangle \subseteq I$.
  Now assume $\langle( f, \mathfrak f)\rangle \subseteq I$.
  Theorem~\ref{thm:rep} there exists a non-zero pseudo-polynomial $(r, \mathfrak f)$, which is minimal with respect to $G$ such that 
  $(f, \mathfrak f) \xrightarrow{G}_+ (r, \mathfrak f)$.
  If $r \neq 0$, then Lemma~\ref{lem:3} shows that
  $\Lt(r, \mathfrak f) \subsetneq \Lt(G) = \Lt(I)$.
  As $\langle (f, \mathfrak f)\rangle \subseteq I$ we also have
  $\langle (r, \mathfrak f)\rangle \subseteq I$ by Lemma~\ref{lem:2} and hence
  $\Lt(r, \mathfrak f) \subseteq \Lt(I)$, a contradiction.
  Thus $(f, \mathfrak f) \xrightarrow{G}_+ 0$.

  (ii) $\Rightarrow$ (iii): Clear from Remark~\ref{rem:1}.

  (iii) $\Rightarrow$ (i): We just have to show that $\Lt(I) \subseteq \Lt(G)$. Let $\langle(f, \mathfrak f)\rangle \subseteq I$ and write $f = \sum_{i \in J} h_i g_i$ with $h_i \in (\mathfrak g_i \mathfrak f^{-1})[x]$ and $\LM(f) = \max_{1 \leq i \leq l}(\LM(h_i g_i))$.
  Thus $\LT(f) = \sum_{i \in J} \LT(h_i)\LT(g_i)$, where $J = \{ i \in J \mid \LM(g_ih_i) = \LM(f_i)\}$.
  Since $\LT(h_i) \in \mathfrak g_i \mathfrak f^{-1}[x]$, for every $\alpha \in \mathfrak f$ we therefore have
  \[ \alpha \LT(f) \subseteq \sum_{i \in J} \mathfrak g_i[x] \LT(g_i), \text{ that is, } \mathfrak \Lt(f, \mathfrak f) = \mathfrak f[x] \LT(f) \subseteq \sum_{i \in I} \Lt(g_i, \mathfrak g_i) = \Lt(G). \qedhere \]
  (iv) $\Leftrightarrow$ (i): This follows from the fact that
  \[ \Lt(G) = \Lt(\{ a_{ij}g_i \mid 1 \leq i \leq l, 1 \leq j \leq n_i\}). \]
\end{proof}

\begin{definition}\label{def:grob}
  Let $I$ be an ideal of $R[x]$. A family $G$ of pseudo-polynomials of $R[x]$ is called a \textit{pseudo-Gröbner basis} of $I$ (with respect to $<$), if $G$ satisfies any of the equivalent conditions of Theorem~\ref{thm:defgroeb}.
\end{definition}

\begin{remark}\label{rem:rem2}
  \hfill
  \begin{enumerate}
    \item
      If one replaces pseudo-polynomials by ordinary polynomials in Theorem~\ref{thm:defgroeb}, one recovers the notion of Gröbner basis of an ideal $I \subseteq G$.
    \item
      Since $R$ is Noetherian, an ideal $I$ of $R[x]$ has a Gröbner basis $\{g_1,\dotsc,g_l\}$ in the ordinary sense~\cite[Corollary 4.1.17]{Adams1994}.
      Recall that his means that $\Lt(g_1,\dotsc,g_l) = \langle \LT(g_1),\dotsc,\LT(g_n) \rangle = \Lt(I)$.
      As $\Lt(g_1,\dotsc,g_l)$ is equal to the leading term ideal of $G = \{ (g_i, R) \mid 1 \leq i \leq l\}$, we see at once that $I$ also has a pseudo-Gröbner basis.
    \item
      In view of Theorem~\ref{thm:defgroeb}~(iv), the notion of pseudo-Gröbner basis is a generalization of~\cite{Adams1997} from the univariate to the multivariate case.
  \end{enumerate}
\end{remark}

Recall that a generating set $G$ of an ideal $I$ in $R[x]$ is called a strong Gröbner basis,
if for every $f \in I$ there exists $g \in G$ such that $\LT(g)$ divides $\LT(f)$.
It is well known, that in case of principal ideal rings, a strong Gröbner basis always exists.
We show that when passing to pseudo-Gröbner bases, we can recover this property for Dedekind domains.

\begin{definition}
  Let $(f, \mathfrak f)$ and $(g, \mathfrak g)$ be two non-zero pseudo-polynomials in $R[x]$. We say that
  $(f,\mathfrak f)$ \textit{divides} $(g, \mathfrak g)$ if $g \mathfrak g[x] \subseteq f \mathfrak f[x]$.
  Let $I \subseteq R[x]$ be an ideal. A set $G = \{ (g_i,\mathfrak g_i) \mid 1 \leq i \leq l\}$ of pseudo-polynomials in $I$ 
  is a \textit{strong pseudo-Gröbner basis}, if for every pseudo-polynomial $(f, \mathfrak f)$ in $I$ there
  exists $i \in \{1,\dotsc,r\}$ such that $\Lt(g_i, \mathfrak g_i)$ divides $\Lt(f, \mathfrak f)$.
\end{definition}

We now fix non-zero pseudo-polynomials $G = \{(g_i,\mathfrak g_i) \mid 1 \leq i \leq l\}$.
For a subset $J \subseteq \{1,\dotsc,r\}$
we define $x_J = \lcm(\LM(g_i) \mid i \in J)$ and $\mathfrak c_J = \sum_{i \in J}\mathfrak g_i \LC(g_i)$.
Let $1 = \sum_{i \in J} a_i \LC(g_i)$ with $a_i \in \mathfrak c_J^{-1} \mathfrak g_i$ for $i \in J$
and define $f_J = \sum_{i \in J} a_i \frac{x_J}{\LM(g_i)} g_i$.
Note that by construction $\LT(f_J) = x_J$.
Finally recall that $J \subseteq \{1,\dotsc,r\}$ is saturated, if for $i \in \{1,\dotsc,r\}$ with $\LM(g_i) \mid x_J$ we have $i \in J$.

\begin{theorem}
  Assume that $G = \{(g_i,\mathfrak g_i) \mid 1 \leq i \leq l\}$ is a pseudo-Gröbner basis of the ideal $I \subseteq R[x]$.
  Then 
  \[ \{(f_J, \mathfrak c_J) \mid J \subseteq \{1,\dotsc,r\} \text{ saturated} \} \]
  is a strong pseudo-Gröbner basis of $I$.
\end{theorem}

\begin{proof}
  Let $(f, \mathfrak f)$ be a non-zero pseudo-polynomial in $I$ and let $J = \{ i \in \{1,\dotsc,r\} \mid \LM(g_i) \text{ divides } \LM(f) \}$.
  Then $J$ is saturated and since $G$ is a pseudo-Gröbner basis of $I$ we have
  \[ \LC(f, \mathfrak f) = \LC(f)\mathfrak f \subseteq \sum_{i \in J} \LC(g_i)\mathfrak g_i = \mathfrak c_J = \LC(f_J\mathfrak c_J). \]
  Furthermore $\LM(f_J) = x_J \mid \LM(f)$ and thus $(f_J,\mathfrak c_J)$ divides $(f, \mathfrak f)$ by Lemma~\ref{lem:div}.
\end{proof}

\begin{corollary}
  Every ideal $I$ of $R[x]$ has a strong pseudo-Gröbner basis.
\end{corollary}

\section{Syzygies}

We already saw in Remark~\ref{rem:rem2}~(ii), that the existence of pseudo-Gröbner basis is a trivial consequence of the fact the Gröbner bases exists whenever the coefficient ring is Noetherian.
The actual usefulness of pseudo-polynomials come from the richer structure of their syzygies, which can be used to characterize and compute Gröbner bases (see~\cite{Moller1988}).
In this section we will show that, similar to the case of principal ideal rings, the syzygy modules of pseudo-polynomials have a basis corresponding to generalized $S$-polynomials.

\subsection{Generating sets}

Consider a family $G = \{(g_i, \mathfrak g_i) \mid 1 \leq i \leq l\}$ of non-zero pseudo-polynomials. As $G = \sum_{1 \leq i \leq l} \mathfrak g_i[x] g_i$, the map
\begin{align*} \varphi \colon \mathfrak g_1[x] \times \dots \times \mathfrak g_l[x] \longrightarrow I, \quad  (h_1,\dots,h_n) \longmapsto \sum_{i=1}^l h_i g_i
\end{align*}
is a well-defined surjective morphism of $R[x]$-modules.

\begin{definition}
  With the notation of the preceding paragraph we call $\ker(\varphi)$ the \textit{syzygies} of $G$ and denote it by $\Syz(G)$.
  A \textit{pseudo-syzygy} of $G$ is a pseudo-element of $\Syz(G)$, that is, a pair $((h_1,\dotsc,h_l), \mathfrak h)$ consisting of
  polynomials $(h_1,\dotsc,h_n) \in K[x]^l$ such that $\mathfrak h \cdot (h_1,\dotsc,h_l) \subseteq \Syz(G)$.
  Equivalently, $\sum_{1 \leq i \leq l} h_i g_i = 0$ and $\mathfrak h h_i \subseteq \mathfrak g_i[x]$ for all $1 \leq i \leq l$.
  
  Assume that the polynomials $g_1, \dotsc,g_l$ are terms. Then we call the pseudo-syzygy $((h_1,\dotsc,h_l), \mathfrak h)$ \textit{homogeneous} if $h_i$ is a term for $1 \leq i \leq l$ and there exists $\alpha \in \N^n$ with $\LM(h_i g_i) = x^\alpha$ for all $1 \leq i \leq l$.
\end{definition}

In the following we will denote by $e_i \in K[x]^l$ the element with components $(\delta_{ij})_{1 \leq j \leq l}$, where $\delta_{ii} = 1$ and $\delta_{ij} = 0$ if $j \neq i$.

\begin{lemma}\label{lem:syzgenset}
  Let $G = \{(g_i,\mathfrak g_i) \mid 1 \leq i \leq l\}$ be non-zero pseudo-polynomials. Then $\Syz(G)$ has a finite generating set of homogeneous pseudo-syzygies.
\end{lemma}

\begin{proof}
  Since $R$ is Noetherian, so is $R[x]$ by Hilbert's basis theorem. In particular $R[x]^l$ is a Noetherian $R[x]$-module. Since the $\mathfrak g_i$ are fractional $R$-ideals, there exists $\alpha \in R$ such that $\alpha \mathfrak g_i \subseteq R$ for all $1 \leq i \leq l$. In particular $\mathfrak g_i[x] \subseteq (\frac 1 \alpha R)[x] = (\frac 1 \alpha) R[x]$.
  Thus
  \[ \mathfrak g_1[x] \times \dots \times \mathfrak g_l[x] \subseteq (1/\alpha) (R[x])^l \cong R[x]^l \]
  is a Noetherian $R[x]$-module as well. Thus the $R[x]$-submodule $\Syz(G)$ is finitely generated.
  A standard argument shows that $\Syz(G)$ is generated by finitely many homogeneous syzygies $v_1,\dotsc,v_m \in \Syz(G)$.
  Hence $\Syz(G) = \langle (v_1,R), \dotsc, (v_m, R)\rangle$ is generated by finitely many homogeneous pseudo-syzygies.
\end{proof}

We can now characterize pseudo-Gröbner bases in terms of syzygies.

\begin{theorem}\label{thm:groebcond2}
  Let $G = \{(g_i,\mathfrak g_i) \mid 1 \leq i \leq l\}$ be non-zero pseudo-polynomials of $R[x]$ and $B$ a finite generating set of homogeneous syzygies of $\Syz(\LT(g_1,\mathfrak g_1),\dots,\LT(g_l,\mathfrak g_l))$.
  Then the following are equivalent:
  \begin{enumerate}
    \item
      $G$ is a Gröbner basis of $\langle G \rangle$.
    \item
      For all $((h_1,\dotsc,h_l),\mathfrak h) \in B$ we have $(\sum_{1 \leq i \leq l} h_i g_i, \mathfrak h) \xrightarrow{G}_+ (0, \mathfrak h)$.
  \end{enumerate}
\end{theorem}

\begin{proof}
  (i) $\Rightarrow$ (ii): Since $\mathfrak h h_i \subseteq \mathfrak g_i[x]$ by definition, we know that
  $\mathfrak h (\sum_{1 \leq i \leq l} h_i g_i) \subseteq \sum_{1 \leq i \leq} \mathfrak g_i[x] \cdot g_i = \langle G \rangle$.
  Hence the element reduces to zero by Theorem~\ref{thm:defgroeb}~(ii).

  (ii) $\Rightarrow$ (i):
  We show that $G$ is a Gröbner basis by verifying Theorem~\ref{thm:defgroeb}~(iii).
  To this end, let $(f, \mathfrak f)$ be a pseudo-polynomial contained in $\langle G \rangle$.
  By Lemma~\ref{lem:rep} there exist elements $u_i \in (\mathfrak g_i \mathfrak f^{-1})[x]$, $1 \leq i \leq l$, such that
  $f = \sum_{1 \leq i \leq l} u_i g_i$.
  We need to show that there exists such a linear combination with $\LM(f) = \max_{1 \leq i \leq l} \LM(u_i g_i)$.
  Let $x^\alpha = \max_{1 \leq i \leq l} (\LM(u_i g_i))$ with $\alpha \in \N^n$,
  and assume that $x^\alpha > \LM(f)$. We will show that $f$ has a representation with strictly smaller degree.
  Denote by $S$ the set $\{ 1 \leq i \leq l \mid \LM(u_ig_i) = x^\alpha\}$.
  As $x^\alpha > \LM(f)$ we necessarily have $\sum_{1 \leq i \leq l} \LT(u_i) \LT(g_i) = 0$.
  In particular $(\sum_{i \in S} e_i \LT(u_i), \mathfrak f)$ is a homogeneous pseudo-syzygy of $\Syz(\LT(g_1,\mathfrak g_1),\dots,\LT(g_l,\mathfrak g_l))$ (since $\mathfrak f \cdot \LT(u_i) \subseteq \mathfrak g_i[x]$).

  Let now $B = ((h_{1j},\dots,h_{lj}), \mathfrak h_j)$, $1 \leq j \leq r$ be the finite generating set of homogeneous pseudo-syzygies.
  By Lemma~\ref{lem:rep} we can find $f_j \in (\mathfrak h_j \mathfrak f^{-1})[x]$ with
  \[ \sum_{i \in S} e_i \LT(u_i) = \sum_{j=1}^r f_j \sum_{i=1}^l e_i h_{ij}. \]
  Since each $\LT(u_i)$ is a term, we may assume that each $f_j$ is also a term.
  Thus for all $1 \leq i \leq l, 1 \leq j \leq r$ we also have
  \[ x^\alpha = \LM(u_ig_i) = \LM(u_i)\LM(g_i) = \LM(f_j)\LM(h_{ij})\LM(g_i) .\]
  whenever $f_j h_{ij}$ is non-zero.
  By assumption, for all $1 \leq j \leq r$ the pseudo-polynomial $(\sum_{1 \leq i \leq l} h_{ij} g_i, \mathfrak h_j)$ reduces to zero with respect to $G$.
  Hence by Theorem~\ref{thm:rep} we can find $v_{ij} \in (\mathfrak
  g_i\mathfrak h_j^{-1})[x]$, $1 \leq i \leq l$, such that $\sum_{1 \leq i \leq l}
  h_{ij}g_i = \sum_{1 \leq i \leq l} v_{ij}g_i$ and
  
  \begin{align}\label{eq} \max_{1 \leq i \leq l} \LM(v_{ij}g_i) = \LM\Bigl(\sum_{i=1}^l h_{ij} gi\Bigr) < \max_{1 \leq i \leq l} \LM(h_{ij} g_i). \end{align}
  The last inequality follows from $\sum_{1 \leq i \leq l} h_{ij} \LT(g_i) = 0$.
  For the element $f$ we started with this implies
  \[ f = \sum_{i = 1}^l u_i g_i = \sum_{i\in S} \LT(u_i)g_i + \sum_{i \in S} (u_i - \LT(u_i)) g_i + \sum_{i \not\in S} u_i g_i. \]
  The first term is equal to
  \[ \sum_{i\in S} \LT(u_i)g_i = \sum_{j = 1}^r f_j \sum_{i = 1}^l h_{ij} g_i = \sum_{j = 1}^r \sum_{i = 1}^l f_j h_{ij}g_i = \sum_{j=1}^r \sum_{i=1}^l f_j v_{ij} g_i = \sum_{i=1}^l (\sum_{j=1}^r f_j v_{ij})g_i.\]
  Now $f_j \in (\mathfrak h_j \mathfrak f^{-1})[x]$, $v_{ij} \in (\mathfrak g_i \mathfrak h_j^{-1})[x]$, hence $f_j v_{ij} \in (\mathfrak g_i \mathfrak f^{-1})[x]$.
  Moreover from~(1) we have
  \[ \max_{i, j} \LM(f_j)\LM(v_{ij})\LM(g_i) < \max_j \max_i \LM(f_j)\LM(h_{ij}g_i) = x^\alpha. \]
  Thus we have found polynomials $\tilde u_i \in (\mathfrak g_i\mathfrak f^{-1})[x]$, $1 \leq i \leq l$, such that
  $\max_{1 \leq i \leq l} \LM(\tilde u_i g_i) < \max_{1 \leq i \leq l} \LM(u_i g_i)$ and $f = \sum_{1 \leq i \leq l}\tilde u_i g_i$. 
\end{proof}

%

\begin{prop}\label{prop:redtomonic}
  Let $(g_i,\mathfrak g_l)_{1 \leq i \leq l}$ be non-zero pseudo-polynomials of
  $R[x]$ and $(a_i)_{1 \leq i \leq l} \in (K^\times)^l$. Consider the map $\Phi \colon
  K[x]^l \to K[x]^l$, $\sum_{1 \leq i \leq l} e_i h_i \mapsto \sum_{1\leq i
  \leq l} e_i \frac{h_i}{a_i}$. Then the following hold:
  \begin{enumerate}
    \item
      The restriction of $\Phi$ induces an isomorphism
      \[ \Syz((g_1,\mathfrak g_1),\dotsc,(g_l,\mathfrak g_l)) \longrightarrow \Syz\left(\Bigl(a_1g_1,\frac{\mathfrak g_1}{a_1}\Bigr),\dotsc,\Bigl(a_l g_l, \frac{\mathfrak g_l}{a_l}\Bigr)\right) \]
      of $R[x]$-modules.
    \item
      If $(h, \mathfrak h)$ is a pseudo-syzygy of $\Syz((g_i,\mathfrak g_i) \mid 1 \leq i \leq l)$, then
      $(\Phi(h), \mathfrak h)$ is a pseudo-syzygy of $\Syz((a_i g_i, \frac{\mathfrak g_i}{a_i}) \mid 1 \leq i \leq l)$ and $\Phi(\langle (h, \mathfrak h) \rangle) = \langle(\Phi(h), \mathfrak h)\rangle$.
  \end{enumerate}
\end{prop}

\begin{proof}
  (i): The map $\Phi$ is clearly $K[x]$-linear. We now show that the image of the syzygies $\Syz((g_i,\mathfrak g_i)_{1 \leq i \leq l})$ under $\Phi$ is contained in $\Syz((a_i g_i, \frac{\mathfrak g_i}{a_i})_{1 \leq i \leq l})$. To this end let $(h_1,\dotsc,h_l) \in \Syz((g_i, \mathfrak g_i)_{1 \leq i \leq l})$, that is, $\sum_{1 \leq i \leq l} h_i g_i = 0$ and $h_i \in \mathfrak g_i[x]$. But then $\sum_{1 \leq i \leq l} \frac{h_i}{a_i} a_i g_i = 0$ and $\frac{h_i}{a_i} \in (\frac{\mathfrak g_i}{a_i})[x]$, that is, $(\frac{h_1}{a_1},\dotsc,\frac{h_l}{a_l}) \in \Syz((a_i g_i, \frac{\mathfrak g_i}{a_i})_{1 \leq i \leq l})$.
  As the inverse map is given by $(h_1,\dotsc,h_l) \mapsto (a_1 h_1,\dotsc,a_l h_l)$, the claim follows.
  (ii): Follows at one from (i)..
\end{proof}

\subsection{Buchberger's algorithm}

\begin{theorem}\label{thm:syzgen}
  Let $(a_i x^{\alpha_i}, \mathfrak g_i)_{1 \leq i \leq l}$ be non-zero pseudo-polynomials, where each polynomial is a term. For $1 \leq i, j \leq l$ we define the pseudo-element
  \[ s_{ij} =\left(\left( \frac{\lcm(x^{\alpha_i}, x^{\alpha_j})}{x^{\alpha_i}} \frac{1}{a_i} e_i - \frac{\lcm(x^{\alpha_i}, x^{\alpha_j})}{x^{\alpha_j}} \frac{1}{\alpha_j} e_j \right), (a_i \mathfrak g_i \cap \alpha_j \mathfrak g_j)\right) \]
   of $K[x]^l$ and for $1 \leq k \leq l$ we set $S_k = \Syz((a_i x^{\alpha_i}, \mathfrak g_i)_{1 \leq i \leq k})$.
  Then the following hold:
  \begin{enumerate}
    \item
      For $1 \leq i, j \leq l$, $i \neq j$, the syzygies $\Syz((a_i x^{\alpha_i}, \mathfrak g_i), (\alpha_jx^{\alpha_j}, \mathfrak g_j))$ are generated by $s_{ij}$.
    \item
      If $B_{k - 1}$ is a generating set of pseudo-generators for $S_{k - 1}$, then
      \[ B = \{ ((h, 0), \mathfrak h) \mid (h, \mathfrak h) \in B_{k - 1} \} \cup \{ s_{ik} \mid 1 \leq i \leq k - 1\} \]
      is a generating set of pseudo-generators for $S_k$.
  \end{enumerate}
\end{theorem}

\begin{proof}
  By Proposition~\ref{prop:redtomonic} we are reduced to the monic case, that is, $a_i = 1$ for $1 \leq i \leq l$.

  (i): It is clear that $s_{ij}$ is a pseudo-syzygy of $((x^{\alpha_i}, \mathfrak g_i), (x^{\alpha_j}, \mathfrak g_j))$.
  Let now $((h_i, h_j), \mathfrak h)$ be a homogeneous pseudo-syzygy with $h_i =
  b_i x^{\beta_i}$, $h_j = b_j x^{\beta_j}$, $\mathfrak h h_i \subseteq \mathfrak
  g_i$ and $\mathfrak h h_j \subseteq \mathfrak g_j$.  We may further assume
  that $b_i \neq 0 \neq b_j$. In particular $x^{\alpha_i}x^{\beta_i} = x^{\alpha_j} x^{\beta_j}$
  and we can write $x^{\beta_i} = \lcm(x^{\alpha_i}, x^{\alpha_j})/x^{\alpha_i} \cdot x^\beta$, $x^{\beta_j} =
  \lcm(x^{\alpha_i}, x^{\alpha_j})/x^{\alpha_j} \cdot x^{\beta}$ for some monomial $x^\beta$.
  We obtain
  \[(h_i, h_j) = x^\beta \left(b_i \frac{\lcm(x^{\alpha_i}, x^{\alpha_j})}{x^{\alpha_i}}, b_j \frac{\lcm(x^{\alpha_i}, x^{\alpha_j})}{x^{\alpha_j}}\right)
  = x^\beta b_i \left(\frac{\lcm(x^{\alpha_i}, x^{\alpha_j})}{x^{\alpha_i}}, - \frac{\lcm(x^{\alpha_i}, x^{\alpha_j})}{x^{\alpha_j}}\right), \]
  where the last equality follows from $b_i + b_j = 0$.
  As $b_i \mathfrak h \subseteq \mathfrak g_i$, $b_i \mathfrak h = b_j \mathfrak h \subseteq \mathfrak g_j$ we obtain $b_i \mathfrak h \subseteq \mathfrak g_i \cap \mathfrak g_j$.
  Thus $\langle ((h_i, h_j), \mathfrak h) \rangle_{R[x]} \subseteq \langle s_{ij} \rangle_{R[x]}$.
  The claim now follows from Lemma~\ref{lem:syzgenset}.

  (ii):
  We start again with a homogeneous pseudo-syzygy $(h, \mathfrak h) = ((h_1,\dotsc,h_{k}), \mathfrak h)$ of $S_k$ of degree $x^\beta$.
  We write $h_i = b_i x^{\beta_i}$ with $\mathfrak h \mathfrak b_i \subseteq \mathfrak g_i$, $1 \leq i \leq k$.
  Since in case $b_k = 0$ we have that $((h_1,\dotsc,h_{k-1}), \mathfrak h)$ is a pseudo-syzygy in $S_{k-1}$, we can assume that $b_k \neq 0$.
  Let $J = \{ i \mid 1 \leq i \leq k - 1, \ b_i \neq 0\}$.
  Since $(h, \mathfrak h)$ is homogeneous, we have $x^{\beta_i} x^{\alpha_i} = x^\beta$ for all $i \in J \cup \{ k \}$ and in particular $\lcm(x^{\alpha_i}, x^{\alpha_k}) \mid x^\beta$ for all $i \in J$.
  Furthermore we have $b_k = - \sum_{i \in J} b_i \in \sum_{i \in J} \langle b_i \rangle_R$ and hence
  $\mathfrak h b_k \subseteq \sum_{i \in J} \mathfrak h b_i \subseteq \sum_{i \in J} \mathfrak g_i$.
  Since at the same time it holds that $\mathfrak h b_k \subseteq \mathfrak g_k$, we conclude that
  $\mathfrak h b_k \subseteq (\sum_{i \in J} \mathfrak g_i) \cap \mathfrak g_k = \sum_{i \in J} (\mathfrak g_i \cap \mathfrak g_k)$.
  Hence there exist $c_i \in (\mathfrak g_i \cap \mathfrak b_k) \mathfrak b^{-1}$, $i \in J$, such that $b_k = - \sum_{i \in J} c_i$.
  For $1 \leq i, j \leq k$ let us denote $\lcm(x^{\alpha_i}, x^{\alpha_j})$ by $x^{\alpha_{ij}}$.
  Now as $x^{\beta}/x^{\alpha_{ik}} \cdot x^{\alpha_{ik}}/x^{\alpha_i} = x^{\beta_i}$ we obtain
  \begin{align*}
    b_k x^{\beta_k} e_k = \sum_{i \in J} - c_i \frac{x^{\beta}}{x^{\alpha_k}}  e_k  &= \sum_{i \in J} - c_i \frac{x^{\beta}}{x^{\alpha_{ik}}} \frac{x^{\alpha_{ik}}}{x^{\alpha_k}} e_k 
    \\ &= \sum_{i \in J} c_i \frac{x^\beta}{x^{\alpha_{ik}}} \left( \frac{x^{\alpha_{ik}}}{x^{\alpha_i}} e_i - \frac{x^{\alpha_{ik}}}{x^{\alpha_k}} e_k\right) - \sum_{i \in J} c_i x^{\beta_i} e_i.
  \end{align*} 
  Hence
  \[ h = \sum_{i=1}^l e_i b_i x^{\beta_i} = \sum_{i=1}^{l-1} e_i b_i x^{\beta_i}
                                       - \sum_{i \in J} c_i x^{\beta_i} e_i
                                       + \sum_{i \in J} c_i \frac{x^\beta}{x^{\alpha_{ik}}} \left( \frac{x^{\alpha_{ik}}}{x^{\alpha_i}} e_i - \frac{x^{\alpha_{ik}}}{x^{\alpha_k}} e_k \right). \]
  We set
  \[ \tilde{h} = \sum_{i=1}^{l-1} e_i b_i x^{\beta_i} - \sum_{i \in J} c_i x^{\beta_i} e_i \quad\text{and}\quad  \tilde{\tilde h} =  \sum_{i \in J} c_i \frac{x^\beta}{x^{\alpha_{ik}}} \left( \frac{x^{\alpha_{ik}}}{x^{\alpha_i}} e_i - \frac{x^{\alpha_{ik}}}{x^{\alpha_k}} e_k \right). \]
  By construction, for all $i \in J$ we have $\mathfrak h c_i \subseteq \mathfrak g_i \cap
\mathfrak g_k$. Together with $J \subseteq \{ 1,\dotsc,k-1\}$ this implies
$\langle (\tilde{\tilde h}, \mathfrak h)_{R[x]} \rangle \subseteq \langle
s_{ik} \mid 1 \leq i \leq k - 1 \rangle_{R[x]}$.
  Let $\Phi \colon \sum_{1 \leq i \leq } e_i h_i \mapsto \sum_{1 \leq i \leq l} h_i g_i$.
  As $h, \tilde{\tilde{h}} \in \ker(\Phi)$, the same holds for $\tilde h$.
  Using again the property $\mathfrak h c_i \subseteq \mathfrak g_i \cap \mathfrak g_k \subseteq \mathfrak g_i$ we conclude that $(\tilde h, \mathfrak h)$ is a pseudo-syzygy of $(g_i, \mathfrak g_i)_{1 \leq i \leq l - 1}$. 
  In particular $\langle ((\tilde h, 0), \mathfrak h) \rangle_{R[x]} \subseteq \langle ((h, 0), \mathfrak h) \mid (h, \mathfrak h) \in B_{k - 1}\rangle_{R[x]}$.
 Invoking again Lemma~\ref{lem:syzgenset}, this proves the claim.
\end{proof}

\begin{definition}
  Let $(f, \mathfrak f)$, $(g, \mathfrak g)$ be two non-zero pseudo-polynomials of $R[x]$. We call
  \[ \left(\left( \frac{\lcm(\LM(f), \LM(g))}{\LM(f)} \frac{1}{\LC(f)} f - \frac{\lcm(\LM(f), \LM(g))}{\LM(g)} \frac{1}{\LC(g)} g \right), \LC(f)\mathfrak f \cap \LC(g)\mathfrak g\right) \]
  the \textit{S-polynomial} of $(f, \mathfrak g)$, $(g, \mathfrak g)$ and denote it by
  $\spoly((f, \mathfrak f), (g, \mathfrak g))$.
\end{definition}

We can now give the analogue of the classical Buchberger criterion in the case of Dedekind domains.

\begin{corollary}\label{cor:buchcrit}
  Let $G = \{ (g_i, \mathfrak g_i) \mid 1 \leq i \leq l\}$ be non-zero pseudo-polynomials of $R[x]$.
  Then $G$ is a Gröbner basis of $\langle G \rangle$ if and only if $\spoly((g_i, \mathfrak g_i), (g_j,\mathfrak g_j))$ reduces to $0$ modulo $G$ for all $1 \leq i < j \leq l$.
\end{corollary}

\begin{proof}
  Applying Theorem~\ref{thm:syzgen}~(ii) inductively using (i) as the base case shows that the set $\{ s_{ij} \mid 1 \leq i < j \leq l\}$ is a of homogeneous pseudo-syzygies generating $\Syz(G)$.
  The claim now follows from Theorem~\ref{thm:groebcond2}.
\end{proof}

\begin{algorithm}\label{alg:groeb}
  Given a family $F = (f_i, \mathfrak f_i)_{1 \leq i \leq l}$ of non-zero pseudo-polynomials, the following steps return a Gröbner basis $G$ of $\langle F \rangle$.
\begin{enumerate}
   \item
     We initialize $\tilde G$ as $\{ ((f_i, \mathfrak f_i), (f_j, \mathfrak f_j)) \mid 1 \leq i < j \leq l \}$ and $G = F$.
   \item
     While $\tilde G \neq \emptyset$, repeat the following steps:
     \begin{enumerate}
       \item
         Pick $((f, \mathfrak f), (g, \mathfrak g)) \in \tilde G$ and compute $(h, \mathfrak h)$ minimal with respect to $G$ such that $\spoly((f, \mathfrak f), (g, \mathfrak g)) \xrightarrow{G}_+ (h, \mathfrak h)$.
       \item
         If $h \neq 0$, set $\tilde G = \tilde G \cup \{((f, \mathfrak f), (h, \mathfrak h)) \mid (f, \mathfrak f) \in G \}$ and $G = G \cup \{ (h, \mathfrak h) \}$.
     \end{enumerate}
   \item
     Return $G$.
\end{enumerate}
\end{algorithm}

\begin{proof}[Algorthm~\ref{alg:groeb} is correct]
  By Corollary~\ref{cor:buchcrit} it is sufficient to show that the algorithm terminates.
  But termination follows as in the field case by considering the ascending chain of leading
  term ideals $\Lt(G)$ (in the Noetherian ring $R[x]$) and using Lemma~\ref{lem:3}.
\end{proof}

\subsection{Product criterion}

For Gröbner basis computations a bottleneck of Buchberger's algorithm
is the reduction of the $S$-polynomials and the number of $S$-polynomials one has to consider.
Buchberger himself gave criteria under which certain $S$-polynomials will reduce to $0$.
In \cite{Moller1988, Lichtblau2012} they have been adapted to coefficient rings that are principal ideal rings and Euclidean domains respectively.
We will now show that the product criterion can be easily translated to the setting of pseudo-Gröbner bases.
Recall that in the case $R$ is a principal ideal domain, the product criterion reads as follows: If $f, g$ are non-zero polynomials in $R[x]$ such that $\GCD(\LC(f), \LC(g)) = 1$ and $\GCD(\LM(f), \LM(g)) = 1$, then the $S$-polynomial $\spoly(f, g)$ reduces to zero modulo $\{f, g\}$.

\begin{theorem}
  Let $(f, \mathfrak f)$, $(g, \mathfrak g)$ be pseudo-polynomials of $R[x]$ such that $\LM(f)$ and $\LM(g)$ are coprime in $K[x]$ and $\LC(f, \mathfrak f)$ and $\LC(g, \mathfrak g)$ are coprime ideals of $R$. Then the $S$-polynomial $\spoly((f, \mathfrak f), (g, \mathfrak g))$ reduces to $0$ modulo $\{(f, \mathfrak f), (g, \mathfrak g)\}$.
\end{theorem}

\begin{proof}
  Denote by $f'$ and $g'$ the tails of $f$ and $g$ respectively.
  We consider three cases.

  In the first case, let both $f$ and $g$ be terms. Then their $S$-polynomial will be $0$ be definition.

  Consider next the case in which $f$ is a term and $g$ is not.
  Then a quick calculation shows that
  \[ (s, \mathfrak s) = \spoly((f, \mathfrak f), (g, \mathfrak g)) = \left(-\frac{1}{\LC(f)\LC(g)} g' f, \LC(f) \mathfrak f \cdot \LC(g) \mathfrak g)\right). \]
  We want to show that $(s, \mathfrak s)$ reduces modulo $\{(f, \mathfrak f)\}$.
  Since $\LM(f)$ divides $\LM(h)$ by definition it is sufficient to show that $\LC(s, \mathfrak s) \subseteq \LC(f, \mathfrak f)$, which is equivalent to
  $\LC(g') \LC(f) \mathfrak f \mathfrak g \subseteq \LC(f) \mathfrak f$.
  But this follows from $\LC(g')\mathfrak g \subseteq R$.
  Hence $(s, \mathfrak s)$ reduces modulo $(f, \mathfrak f)$ to
  \[ \left(s - \frac{\LT(g')}{\LC(g)\LC(f)} f, \LC(f)\mathfrak f \cdot \LC(g) \mathfrak g\right) = \left(-\frac{1}{\LC(f)\LC(g)}(g' - \LT(g'))f, \LC(f)\mathfrak f \cdot \LC(g) \mathfrak g\right). \]
  Applying this procedure recursively, we see that $(s, \mathfrak s)$ reduces to $0$ modulo $\{(f, \mathfrak f)\}$.

  Now consider the case, where $f$ and $g$ are both not terms, that is, $f' \neq 0 \neq g'$.
  Then the $S$-polynomial of $(f, \mathfrak f)$ and $(g, \mathfrak g)$ is equal to
  \[ (s, \mathfrak s) = \left( (\frac{\LT(g)}{\LC(f)} f - \frac{\LT(f)}{\LC(g)} g), \LC(f) \mathfrak f \LC(g) \mathfrak g \right)= (\frac{1}{\LC(f)\LC(g)} (f' g - g' f), \LC(f) \mathfrak f \LC(g) \mathfrak g). \]
  Since $\LM(f)$ and $\LM(g)$ are coprime, we have $\LM(f'g) \neq \LM(g'f)$ and therefore $\LM(s)$ is either $\LM(f' g)$ or $\LM(g' f)$.
  In particular $\LM(s)$ is either a multiple of $\LM(f)$ or $\LM(g)$.
  If $\LM(s) = \LM(g' f)$
  then $\LC(s) = \-\LC(g')/\LC(g)$ and $\LC(s, \mathfrak s) = \LC(g') \LC(f) \mathfrak f \mathfrak g$.
  As in third case, $(s, \mathfrak s)$ reduces to
  \begin{align*} &\left(- \frac{1}{\LC(f)\LC(g)} (f' g - g' f) - \frac{\LT(g')}{\LC(g)\LC(f)} f, \LC(f)\mathfrak f \LC(g)\mathfrak g\right)\\
    = &\left(-\frac{1}{\LC(f)\LC(g)} ( f' g - (g' - \LT(g')))f, \LC(f)\mathfrak f \LC(g)\mathfrak g\right),
  \end{align*}
  and similar in the other case.
  Note that again, the leading monomial of $(f'g - (g' - \LT(g')) f)$ is a multiple of $\LM(f)$ and $\LM(g)$.
  Inductively this shows that $(s, \mathfrak s) \xrightarrow{\{(f,\mathfrak f), (g, \mathfrak g)\}}_+ 0$.\qedhere
\end{proof}

\section{Coefficient reduction}

Although in contrast to $\Q[x]$ the naive Gröbner basis computation of
an ideal $I$ of $\Z[x]$ is free of denominators, the problem of quickly growing
coefficients is still present.
In case a non-zero element $N \in I \cap \Z$ is known this problem can be
avoided: By adding $N$ to the generating set under consideration, all
intermediate results can be reduced modulo $N$, leading to tremendous improvements in runtime, see ~\cite{Eder2018}.

In this section we will describe a similar strategy for the computation of
pseudo-Gröbner bases in case the coefficient ring is the ring of integers of a
finite number field. Although this is
quite similar to the integer case, we now have to deal with the growing size of
the coefficients of polynomials themselves as well as with the size of the
coefficient ideals.

\subsection{Admissible reductions}
We first describe the reduction operations that are allowed during a Gröbner
basis computation for arbitrary Dedekind domains.

\begin{prop}\label{prop:red}
  Let $R$ be a Dedekind domain, and $(f, \mathfrak f)$ a non-zero pseudo-polynomials of $R[x]$.
  \begin{enumerate}
    \item
      If $(g, \mathfrak g)$ is a pseudo-polynomial of $R[x]$ with $\mathfrak f
    f = \mathfrak g g$, then $(f, \mathfrak f)$ reduces to $0$ modulo $(g,
  \mathfrak g)$.
    \item
      Write $f = \sum_{1 \leq i \leq d} c_{\alpha_i} x^{\alpha_i}$ with $c_{\alpha_i} \neq 0$.
      Assume that $g = \sum_{1 \leq i \leq d} \bar c_{\alpha_i} x^{\alpha_i} \in R[x]$ is a polynomial and $\mathfrak N$ a fractional ideal of
      $R$ such that $c_{\alpha_i} - \bar c_{\alpha_i} \in
      \mathfrak N \mathfrak f^{-1}$ for $1 \leq i \leq d$.
      Then $f$ reduces to $0$ modulo $((g,
      \mathfrak f), (1, \mathfrak N))$.
  \end{enumerate}
\end{prop}

\begin{proof}
  (i): By assumption $\LM(f) = \LM(g)$. Moreover, as $\frac{\LC(f)}{\LC(g)} \in
  \mathfrak g \mathfrak f^{-1}$ we see that $(f, \mathfrak f)$ reduces to
  \[ \left(f - \frac{\LC(f)}{\LC(g)}\frac{\LM(f)}{\LM(g)} g, \mathfrak f\right) = (0, \mathfrak f). \]
  (ii): 
  We first consider the case that $\LM(f) \neq \LM(g)$. By assumption this implies that $\LC(f) \in \mathfrak N \mathfrak f^{-1}$ and
  $(f, \mathfrak f)$ reduces to $(f - \LC(f)\LM(f), \mathfrak f)$ modulo $(1, \mathfrak N)$.
  Since we also have $(f - \LC(f)\LM(f)) - g \in \mathfrak N\mathfrak f^{-1}[x]$,
  we now may assume that $(f - \LC(f)\LM(f)) = 0$, in which case we are finished, or
  $\LM(f) = \LM(g)$.
  In the latter case, we use $\LC(f) - \LC(g) \in \mathfrak N\mathfrak f^{-1}$ and
  $\LC(f) = 1 \cdot \LC(g) + (\LC(f) - \LC(g))\cdot 1$
  to conclude that $(f, \mathfrak f)$ reduces to $(\tilde f, \mathfrak f)$ modulo $\{(g, \mathfrak f), (1, \mathfrak N)\}$,
  where $\tilde f = f - g - (\LC(f) - \LC(g))\LM(f)$.
  Since the polynomial $\tilde f$ satisfies $\tilde f \in \mathfrak N\mathfrak f^{-1}[x]$, it reduces to $0$ modulo $(1, \mathfrak N)$.
\end{proof}

Since our version of Buchberger's algorithm rests on S-polynomials reducing to $0$ (see Corollary~\ref{cor:buchcrit}),
the previous result immediately implies the correctness of the following modification of Algorithm~\ref{alg:groeb}.

\begin{corollary}\label{cor:red}
  Assume that $F = (f_i,\mathfrak f_i)_{1 \leq i \leq l}$ is family of pseudo-polynomials, such that $\langle F \rangle$ contains a non-zero ideal $\mathfrak N$ of $R$. 
  After adding $(1, \mathfrak N)$ to $F$, in Algorithm~\ref{alg:groeb} include the following Step after (a):
  \begin{enumerate}
    \item[(a')]
      Let $(g_1, \mathfrak g_1)$ be a non-zero pseudo-polynomial
      with $\mathfrak g_1 g_1  = \mathfrak h h$. 
      Now let $g_1 = \sum_{i}c_{\alpha_i}x^{\alpha_i}$ with $c_{\alpha_i} \neq 0$.
      Find a polynomial $g_2 = \sum_{i}\bar c_{\alpha_i} x^{\alpha_i}$ with $c_{\alpha_i} - \bar c_{\alpha_i} \in \mathfrak N\mathfrak g^{-1}[x]$ for all $i$ and replace $(h, \mathfrak h)$ by $(g_2, \mathfrak g_1)$.
  \end{enumerate}
  Then the resulting algorithm is still correct.
\end{corollary}

\subsection{The case of rings of integers.}

It remains to describe how to use the previous results to bound the size of the intermediate pseudo-polynomials.
Since this question is meaningless in the general settings of Dedekind domains, we now restrict to the case where $R$ is the ring of integers of a finite number field $K/\mathbf Q$.
We assume that $I \subseteq R[x]$ is an ideal which contains non-zero ideal $\mathfrak N$ of $R$.
In view of Proposition~\ref{prop:red}, we want to solve the following two problems for a given non-zero pseudo-polynomial $(f, \mathfrak f)$ of $R[x]$.
\begin{enumerate}
  \item
    Find a pseudo-polynomial $(g, \mathfrak g)$ of $R[x]$ with $\mathfrak g$ small such that $\mathfrak f f = \mathfrak g g$.
  \item
    Find a pseudo-polynomial $(g, \mathfrak f)$ of $R[x]$, such that $g$ has small coefficients, every monomial of $g$ is a monomial of $f$, and $f - g \in \mathfrak N \mathfrak f^{-1}[x]$.
\end{enumerate}

We will now translate this problem to the setting of pseudo-elements in
projective $R$-modules of finite rank, where the analogous problems
are already solved in the context of generalized Hermite form algorithms.
To this end, let $f = \sum_{1 \leq i \leq d} c_{\alpha_i} x^{\alpha_i}$, $c_{\alpha_i} \neq 0$, and consider
\[ \pi \colon K[x] \longrightarrow K^d, \, \sum_{\alpha \in \N^n} c_\alpha x^{\alpha} \longmapsto (c_{\alpha_i})_{1 \leq i \leq d}, \quad \iota \colon K^d \longrightarrow K[x], \, (c_{\alpha_i})_{1 \leq i \leq d} \longmapsto \sum_{i=1}^d c_{\alpha_i} x^{\alpha_i} . \]
Using these $K$-linear maps, we can think of pseudo-polynomials having the same
support as $f$ as projective $R$-submodules of $V$ of rank one, that is, as pseudo-elements in $K^d$.
Moreover, if $\mathfrak f \pi(f) = \mathfrak g w$ for some $w \in K^d$ and
fractional ideal $\mathfrak g$ of $R$, then $\mathfrak f f =
\mathfrak g \iota(w)$.
In particular, by setting $v = (v_i)_{1 \leq i \leq d} = \pi(f) \in K^d$, problems~(i) and~(ii) are equivalent to the following two number theoretic problems:
\begin{enumerate}
  \item[(i')] Find a pseudo-element $(w, \mathfrak g)$ of $K^d$ with $\mathfrak g$ small such that $\mathfrak fv = \mathfrak gw$.
  \item[(ii')] Find a pseudo-element $(w, \mathfrak f)$ of $K^d$, such that $w_i$ is small and $v_i - w_i \in \mathfrak N\mathfrak f^{-1}$ for all $1 \leq i \leq d$.
\end{enumerate}

Hence, we can reduce pseudo-polynomials by applying the following two
algorithms to the coefficient ideal and the coefficients respectively.
Both are standard tools in algorithmic algebraic number theory, see \cite{Biasse2017} for a discussion including a complexity analysis.

\begin{lemma}
  Let $\mathfrak N$ be a non-zero ideal of $R$.
  \begin{enumerate}
    \item
      There exists an algorithm, that given a fractional ideal $\mathfrak a$ of $R$ and a vector $v \in K^d$ determines an ideal $\mathfrak b$ of $R$ and a vector $w \in K^d$ such that $\mathfrak a v = \mathfrak b w$ and the norm $\#(R/\mathfrak b)$ can be bounded by a constant that depends only on the field $K$ (and not on $\mathfrak a$ or $v$).
    \item
      There exists an algorithm, that given a non-zero ideal $\mathfrak f$ of $R$ and an element $\alpha$ of $K$, determines an element $\beta \in K$ such that $\alpha - \beta \in \mathfrak N\mathfrak f^{-1}$ and the size of $\beta$ can be bounded by a constant that depends only on the field $K$ and the norms $\#(R/\mathfrak N)$, $\#(R/\mathfrak f)$.
  \end{enumerate}
\end{lemma}

\begin{remark}
  Recall that $\mathfrak N$ is a non-zero ideal of $R$ such that $\mathfrak N \subseteq I$, where $I$ is the ideal of $R[x]$ for which we want to find a pseudo-Gröbner basis.
  The preceding discussion together with Corollary~\ref{cor:red} implies that
  during Buchberger's algorithm (Algorithm~\ref{alg:groeb}), we can reduce the intermediate results so that
the size of all pseudo-polynomials is bounded by a constant depending only on
$\mathfrak N$ and $K$.
\end{remark}

\begin{remark}
  Assume that $I \subseteq R[x]$ is an ideal. Then there exists a non-zero ideal $\mathfrak N$ of $R$ contained in $I$ if and only if $K[x] = \langle I \rangle_{K[x]}$.
  In case, one can proceed as follows to find such an ideal $\mathfrak N$.
  Let $F = (f_i, \mathfrak f_i)_{1 \leq i \leq l}$ be a generating set of pseudo-polynomials of $I$.
  Using classical Gröbner basis computations and the fact that $1 \in \langle I \rangle_{K[x]}$
  we can compute $a_i \in K$, $1 \leq i \leq l$, such that $1 = \sum_{1 \leq i \leq l} a_i f_i$.
  Next we determine $d \in R$ such that $da_i \in \mathfrak f_i$ for all $1 \leq i \leq l$.
  Then
  \[ d = \sum_{i=1}^l da_if_i \in \sum_{i=1}^l \mathfrak f_i f_i \subseteq I \]
  and thus the non-zero ideal $\mathfrak N = dR$ satisfies $\mathfrak N \subseteq I$.
\end{remark}

\section{Applications}

We give a few applications of pseudo-Gröbner bases to classical problems from
algorithmic commutative algebra as well as to the problem of computing primes of bad reduction.

\subsection{Ideal membership and intersections}

\begin{prop}
  Let $I$ be an ideal of $R[x]$ given by a finite generating set of non-zero (pseudo-)polynomials.
  There exists an algorithm, that given a polynomial $f$ respectively a pseudo-polynomial $(f, \mathfrak f)$ decides whether $f \in I$ respectively $\langle (f, \mathfrak f) \rangle \subseteq I$.
\end{prop}

\begin{proof}
  Since $f \in I$ if and only if $\langle (f, R)\rangle \subseteq I$, we can restrict to the case of pseudo-polynomials. 
  After computing a pseudo-Gröbner basis of $I$ using Algorithm~\ref{alg:groeb}, we can use Theorem~\ref{thm:defgroeb}~(ii) to decide membership.
\end{proof}

Next we consider intersections of ideals, where as in the case of fields we use an elimination ordering.

\begin{prop}\label{prop:intersect}
  Consider $R[x, y]$ with elimination order with the $y$ variables larger than the $x$ variables.
  Let $G = \{(g_i,\mathfrak g_i) \mid 1 \leq i \leq l\}$ be a pseudo-Gröbner basis of an ideal $I \subseteq R[x, y]$.
  Then $\{ (g_i, \mathfrak g_i) \mid g_i \in K[x] \}$
  is a pseudo-Gröbner basis of $I \cap R[x]$.
\end{prop}

\begin{proof}
  Follows from Theorem~\ref{thm:defgroeb}~(iv) and the corresponding result for Gröbner bases, see \cite[Theorem 4.3.6]{Adams1994}.
\end{proof}

\begin{corollary}
  Let $I$, $J$ be two ideals of $R[x]$ given by finite generating sets of non-zero (pseudo-)polynomials.
  Then there exists an algorithm that computes a finite generating set of pseudo-polynomials of $I \cap J$.
\end{corollary}

\begin{proof}
  This follows from Proposition~\ref{prop:intersect} and the classical fact that
  $I \cap J = \langle wI, (1 - w)J \rangle_{R[x, w]} \cap R[x]$,
  where $w$ is an additional variable (see~\cite[Proposition 4.3.9]{Adams1994}).
\end{proof}

\begin{corollary}\label{coro:intersectR}
  Let $I \subseteq R[x]$ be an ideal.
  Then there exists an algorithm for computing $I \cap R$. 
\end{corollary}

%

\subsection{Primes of bad reduction}

It seems to be well known, that in the case where $R$ is $\Z$, the primes of bad reduction of a variety can be determined by computing Gröbner bases of ideals corresponding to singular loci.
Due to the lack of references we give a proof of this folklore result and show how it relates to pseudo-Gröbner bases.
Assume that $X \subseteq
\PP_{R}^n$ is a subscheme which is flat over $\Spec(R)$, has smooth generic fiber $X_K$ and is pure of dimension $k$.
Our aim is to determine the primes of bad reduction of $X$, that is, we want to find all points $\mathfrak p \in \Spec(R)$ such that the special fiber $X_{\mathfrak p}$ is not smooth.
By passing to an affine cover, we may assume that $X$ is a closed subscheme $V(f_1,\dotsc,f_l)$ of $\mathbf A_R^n$, where $f_1,\dotsc,f_l \in R[x]$.
Let $\mathfrak p \in \Spec(R)$, $\mathfrak p \neq 0$ and denote by $k_{\mathfrak p} = R/\mathfrak p$ the residue field.
Let $J = (\frac{\partial f_i}{\partial x_j})_{1\leq i \leq l, 1 \leq j \leq n}$ be the Jacobian matrix.

\begin{theorem}
  Let $X = V(f_1,\dotsc,f_l)$ and $I$ the ideal of $R[x]$ generated by $f_1,\dotsc,f_l$ and the $(n-k)$ minors of $J$.
  Then $X_\mathfrak p \subseteq \mathbf A_{k_\mathfrak p}^n$ is smooth if and only if $\mathfrak p$ does not divide $I \cap R$.
\end{theorem}

\begin{proof}
  The flatness condition implies that $X_\mathfrak p$ has dimension $k$. By the Jacobian criterion (\cite[Chapter 4, Theorem 2.14]{Liu2002}, $X_\mathfrak p$ is smooth if and only if $J_\mathfrak p(p)$ has rank $n - k$ for all $p \in X_\mathfrak p (\bar k_\mathfrak p)$, where $J_\mathfrak p = (\frac{\partial \bar f_i}{\partial x_j})_{1 \leq i \leq l, 1 \leq j \leq n}$ is the Jacobian of $\bar f_1,\dotsc,\bar f_l$.
  Thus $X_\mathfrak p$ is smooth if and only if the ideal of $k_\mathfrak p[x]$ generated by $\bar f_1,\dotsc,\bar f_l$ and the $(n - k)$-minors of $J_\mathfrak p$ is equal to $k_\mathfrak p[x]$.
  Hence $X_\mathfrak p$ is smooth if and only if the ideal $(I, \mathfrak p)$ of $R[x]$ is equal to $R[x]$.
  Now $(I, \mathfrak p) \subsetneq R[x]$ if and only if there exists a maximal ideal $M$ of $R[x]$ containing $(I, \mathfrak p)$. But in this case the kernel $R \cap M$ of the projection $R \to R[x]/M$ contains $\mathfrak p$ and must therefore be equal to $\mathfrak p$. As $\mathfrak p \subseteq (I, \mathfrak p) \cap R \subseteq M \cap R = \mathfrak p$, the existence of $M$ is equivalent to $(I, \mathfrak p) \cap R = \mathfrak p$, that is, $I \cap R \subseteq \mathfrak p$.
\end{proof}

Combining this with the previous subsection, the primes of bad reduction can be easily characterized with pseudo-Gröbner bases. Note that this does not determine the primes themselves, since one has to additionally determine the prime ideal factors.

\begin{corollary}
  Let $X = V(f_1,\dotsc,f_k)$ and $I$ the ideal of $R[x]$ generated by $f_1,\dotsc,f_l$ and the $(n-k)$ minors of the Jacobian matrix $J$.
  Let $\{ (g_i,\mathfrak g_i) \mid 1 \leq i \leq l\}$ be a pseudo-Gröbner basis of $I$ and $\mathfrak N = \sum \mathfrak g_i g_i \subseteq R$, where the sum is over all $1 \leq i \leq l$ such that $g_i \in K$. Then $\mathfrak p$ is a prime of bad reduction of $X$ if and only if $\mathfrak p$ divides $\mathfrak N$.
\end{corollary}

\begin{example}
  To have a small non-trivial example, we look at an elliptic curve defined over a number field.
  Although there are other techniques to determine the primes of bad reduction, we will do so
  using pseudo-Gröbner bases.
  Consider the number field $K = \Q(\sqrt{10})$ with ring of integers $\mathcal O_K = \Z[a]$, where  $a = \sqrt{10}$. Let $E/K$ be the elliptic curve defined by 
  \[ f = y^2 - x^2 + (1728a+3348)x + (44928a-324432) \in K[x, y]. \]
  Note that this is a short Weierstrass equation of the elliptic curve with
  label \texttt{6.1-a2} from the LMFDB~(\cite{lmfdb}).
  To determine the places of bad reduction, we consider the ideal 
  \[ I = \left\langle f, \frac{\partial f}{\partial x}, \frac{\partial f}{\partial y} \right\rangle \subseteq \mathcal O_K[x, y]. \]
  Applying Algorithm~\ref{alg:groeb} we obtain a pseudo-Gröbner basis $G$, which together with Corollary~\ref{coro:intersectR} allows us to compute
  \[ I \cap R = \langle 940369969152, 437864693760a +  71663616\rangle \subseteq \mathcal O_K. \]
  The ideal $I \cap R$ has norm $67390312367240773632 = 2^{31} \cdot 3^{22}$ and factors as
  \[ I \cap R = \langle 2, a\rangle^{31} \cdot \langle 3, a + 2 \rangle^{15} \cdot \langle 3, a+4\rangle^7. \]
  Thus the primes of bad reduction are $\langle 2, a \rangle$, $\langle 3, a + 2 \rangle$ and $\langle 3, a + 4 \rangle$.
  Note that the conductor of $E$ is divisible only by $\langle 2, a \rangle$ and $\langle 3, a + 2 \rangle$ (the model we chose is not minimal at $\langle 3, a + 4 \rangle$).
  In fact this can be seen by determining the primes of bad reduction of the model $y^2 = x^3 + \frac 1 3(64a + 124)x + \frac 1{27}(1664a - 12016)$, which is minimal at $\langle 3, a + 4 \rangle$ (computed with \textsc{Magma} \cite{Bosma1997}).
\end{example}

\bibliography{pseudo_grobner}
\bibliographystyle{amsalpha}

\end{document}